\RequirePackage[l2tabu, orthodox]{nag}
\documentclass[final,10pt,reqno]{amsart}

\usepackage{bm, xspace}

\usepackage[final]{microtype}
\makeatletter
\def\MT@register@subst@font{\MT@exp@one@n\MT@in@clist\font@name\MT@font@list
   \ifMT@inlist@\else\xdef\MT@font@list{\MT@font@list\font@name,}\fi}
\makeatother 

\usepackage{mparhack, fixltx2e}
\usepackage{amsmath, amssymb, amsfonts, amsthm, mathtools}
\allowdisplaybreaks
\usepackage[all,warning]{onlyamsmath}
\usepackage{fixmath}
\mathtoolsset{centercolon}

\usepackage[latin1]{inputenc}
\usepackage[T1]{fontenc}

\usepackage[draft=false,pdftex]{hyperref}
\hypersetup{
    bookmarks=true,         
    unicode=false,          
    pdftoolbar=true,        
    pdfmenubar=true,        
    pdffitwindow=false,     
    pdfstartview={FitH},    
    pdftitle={Equilateral sets and a Sch\"utte Theorem for the 4-norm},    
    pdfauthor={Konrad J. Swanepoel},     
    pdfnewwindow=false,      
    colorlinks=true,       
    linkcolor=black,          
    citecolor=black,        
    filecolor=black,      
    urlcolor=black           
}

\theoremstyle{plain}
\newtheorem{theorem}{Theorem}
\newtheorem{lemma}[theorem]{Lemma}
\newtheorem{proposition}[theorem]{Proposition}
\newtheorem{corollary}[theorem]{Corollary}

\newcommand{\setbuilder}[2]{\left\{#1\;\middle|\;#2\right\}}
\newcommand{\set}[1]{\left\{#1\right\}}
\newcommand{\epsi}{\varepsilon}
\newcommand{\norm}[1]{\left\lVert#1\right\rVert}
\newcommand{\abs}[1]{\left\lvert#1\right\rvert}
\newcommand{\numbersystem}[1]{\mathbb{#1}}
\newcommand{\bN}{\numbersystem{N}}
\newcommand{\bR}{\numbersystem{R}}
\newcommand{\vect}[1]{\bm{#1}}
\newcommand{\va}{\vect{a}}
\newcommand{\vb}{\vect{b}}
\newcommand{\vo}{\vect{o}}
\newcommand{\vx}{\vect{x}}
\newcommand{\vy}{\vect{y}}

\newcommand{\dimensional}{\nobreakdash-\hspace{0pt}dimensional\xspace}

\title{Equilateral sets and a Sch\"utte Theorem for the $4$-norm}

\author{Konrad J.\ Swanepoel}
\address{Department of Mathematics, London School of Economics and Political Science, Houghton Street, London WC2A 2AE, United Kingdom.}
\email{k.swanepoel@lse.ac.uk}

\begin{document}
\begin{abstract}
A well-known theorem of Sch\"utte (1963) gives a sharp lower bound for the ratio of the maximum and minimum distances between $n+2$ points in $n$\dimensional Euclidean space.
In this note we adapt B\'ar\'any's elegant proof (1994) of this theorem to the space $\ell_4^n$.
This gives a new proof that the largest cardinality of an equilateral set in $\ell_4^n$ is $n+1$, and gives a constructive bound for an interval $(4-\epsi_n,4+\epsi_n)$ of values of $p$ close to $4$ for which it is known that the largest cardinality of an equilateral set in $\ell_p^n$ is $n+1$.
\end{abstract}

\maketitle

\section{Introduction}
A subset $S$ of a normed space $X$ with norm $\norm{\cdot}$ is called \emph{equilateral} if for some $\lambda>0$, $\norm{\vx-\vy}=\lambda$ for all distinct $\vx,\vy\in S$.
Denote the largest cardinality of an equilateral set in a finite\dimensional normed space $X$ by $e(X)$.

For $p\geq 1$ define the $p$-norm of a vector $\vx=(x_1,\dots,x_n)\in\bR^n$ as
\[\norm{\vx}_p=\norm{(x_1,\dots,x_n)}_p=\biggl(\sum_{i=1}^n \abs{x_i}^p\biggr)^{1/p}\text{.}\]
When dealing with a sequence $\vx_1,\dots,\vx_m\in\bR^n$ of vectors, we denote the coordinates of $\vx_i$ as $(x_{i,1},\dots,x_{i,n})$.
Denote the normed space $\bR^n$ with norm $\norm{\cdot}_p$ by $\ell_p^n$.
It is not difficult to find examples of equilateral sets showing that $e(\ell_p^n)\geq n+1$.
It is a simple exercise in linear algebra to show that $e(\ell_2^n)\leq n+1$.
A problem of Kusner \cite{Guy} asks if the same is true for $\ell_p^n$, where $p>1$.
For the current best upper bounds on $e(\ell_p^n)$, see \cite{Alon-Pudlak}.
We next mention only the results that decide various cases of Kusner's question.
A compactness argument gives for each $n\in\bN$ the existence of $\epsi_n>0$ such that $p\in(2-\epsi_n,2+\epsi_n)$ implies $e(\ell_p^n)=n+1$.
However, this argument gives no information on $\epsi_n$.
As observed by Cliff Smyth \cite{Smyth}, the following theorem of Sch\"utte \cite{Schutte} can be used to give an explicit lower bound to $\epsi_n$ in terms of $n$:
\begin{theorem}[Sch\"utte \cite{Schutte}]
Let $S$ be a set of at least $n+2$ points in $\ell_2^n$.
Then
\[ \frac{\displaystyle\max_{\vx,\vy\in S}\norm{\vx-\vy}_2}{\displaystyle\min_{\vx,\vy\in S, \vx\neq\vy}\norm{\vx-\vy}_2} \geq
\begin{cases}
\displaystyle\biggl(1+\frac{2}{n}\biggr)^{1/2} & \text{if $n$ is even,}\\
\displaystyle\biggl(1+\frac{2}{n-(n+2)^{-1}}\biggr)^{1/2} & \text{if $n$ is odd.}
\end{cases}
\]
\end{theorem}
The lower bounds in this theorem are sharp.
\begin{corollary}[Smyth \cite{Smyth}]\label{smythcor}
If $\abs{p-2} < \frac{2\log(1+2/n)}{\log(n+2)}=\frac{4(1+o(1))}{n\log n}$ then the largest cardinality of an equilateral set in $\ell_p^n$ is $e(\ell_p^n)=n+1$.
\end{corollary}
The dependence of $\epsi_n=\frac{4(1+o(1))}{n\log n}$ on $n$ is necessary, since $e(\ell_p^n)>n+1$ if $1\leq p<2-\frac{1+o(1)}{(\ln 2) n}$ \cite{Swanepoel-AdM}.
(These are the only known cases where the answer to Kusner's question is negative.)

There is also a linear algebra proof that $e(\ell_4^n)=n+1$ \cite{Swanepoel-AdM}.
As in the case of $p=2$, compactness gives an ineffective $\epsi_n>0$ such that if $p\in(4-\epsi_n,4+\epsi_n)$, then $e(\ell_p^n)=n+1$.
The question arises whether Sch\"utte's theorem can be adapted to $\ell_4^n$, so that a conclusion similar to Corollary~\ref{smythcor} can be made for $p$ close to $4$.
Proofs of Sch\"utte's theorem have been given by Sch\"utte \cite{Schutte}, Schoenberg \cite{Schoenberg}, Seidel \cite{Seidel} and B\'ar\'any \cite{Barany}.
It is the purpose of this note to show that B\'ar\'any's simple and elegant proof of Sch\"utte's theorem can indeed be adapted.

\begin{theorem}\label{maintheorem}
Let $S$ be a set of at least $n+2$ points in $\ell_4^n$.
Then
\[ \frac{\displaystyle\max_{\vx,\vy\in S}\norm{\vx-\vy}_4}{\displaystyle\min_{\vx,\vy\in S, \vx\neq\vy}\norm{\vx-\vy}_4} \geq
\begin{cases}
\displaystyle\biggl(1+\frac{2}{n}\biggr)^{1/4} & \text{if $n$ is even,}\\
\displaystyle\biggl(1+\frac{2}{n-(n+2)^{-1}}\biggr)^{1/4} & \text{if $n$ is odd.}
\end{cases}
\]
\end{theorem}
\begin{corollary}\label{cor}
If $\abs{p-4} < \frac{4\log(1+2/n)}{\log(n+2)} = \frac{8(1+o(1))}{n\log n}$ then the largest cardinality of an equilateral set in $\ell_p^n$ is $e(\ell_p^n)=n+1$.
\end{corollary}
We do not know whether the lower bounds in Theorem~\ref{maintheorem} are sharp.
The following is the best upper bound that we can show.
\begin{proposition}\label{construction}
There exists a set $S$ of $n+2$ points in $\ell_4^n$ such that
\[ \frac{\displaystyle\max_{\vx,\vy\in S}\norm{\vx-\vy}_4}{\displaystyle\min_{\vx,\vy\in S, \vx\neq\vy}\norm{\vx-\vy}_4} = 1+\sqrt{\frac{2}{n}}+O(n^{-3/4}).\]
\end{proposition}
Unfortunately, this bound is far from the lower bound of $1+\frac{1}{2n}+O(n^{-2})$ given by Theorem~\ref{maintheorem}.
\section{Proofs}

\begin{proof}[Proof of Theorem~\ref{maintheorem}]
Consider any $\vx_1,\dots,\vx_{n+2}\in\bR^n$, and let
\[ \mu=\min_{i\neq j}\norm{\vx_i-\vx_j}_4\]
and
\[ M=\max_{i,j}\norm{\vx_i-\vx_j}_4\text{.}\]
By Radon's theorem \cite{Barvinok} there is a partition $A\cup B$ of $\{\vx_1,\dots,\vx_{n+2}\}$ such that the convex hulls of $A$ and $B$ intersect.
Without loss of generality we may translate the points so that $\vo$ lies in both convex hulls.
Write $A=\{\va_1,\dots,\va_K\}$ and $B=\{\vb_1,\dots,\vb_L\}$ where $K+L=n+2$ and $K,L\geq 1$.
Then there exist $\alpha_1,\dots,\alpha_K,\beta_1,\dots,\beta_L\geq 0$ such that
\begin{equation}\label{1}
\left.
\begin{aligned}
\sum_{i=1}^K\alpha_i&=1\text{,} & \sum_{i=1}^K\alpha_i \va_i &= \vo\text{,}\\
\sum_{j=1}^L\beta_j&=1\text{,} & \sum_{j=1}^L\beta_j \vb_j &= \vo\text{.}
\end{aligned}
\right\}
\end{equation}
Also, for all $i\in[K]$ and $j\in[L]$,
\begin{align}
\norm{\va_i-\va_j}_4^4&\leq M^4 \quad\text{whenever $i\neq j$,}\label{2}\\
\norm{\vb_i-\vb_j}_4^4&\leq M^4 \quad\text{whenever $i\neq j$,}\label{3}\\
\text{and}\quad\norm{\va_i-\vb_j}_4^4&\geq \mu^4\text{.}\label{4}
\end{align}
Apply the operation $\displaystyle\sum_{i=1}^K\alpha_i\sum_{\substack{j=1\\ j\neq i}}^K\alpha_j$ to both sides of the inequality \eqref{2}:
\begin{align*}
& \phantom{{}={}} \biggl(1-\sum_{i=1}^K\alpha_i^2\biggr)M^4 = \sum_{i=1}^K\alpha_i(1-\alpha_i)M^4 = \sum_{i=1}^K\alpha_i\sum_{\substack{j=1\\ j\neq i}}^K\alpha_j M^4\\
&\geq \sum_{i=1}^K\alpha_i\sum_{j=1}^K\alpha_j \sum_{m=1}^n(a_{i,m}-a_{j,m})^4\\
&= \sum_{m=1}^n\sum_{i=1}^K\sum_{j=1}^K\alpha_i\alpha_j (a_{i,m}^4-4a_{i,m}^3a_{j,m}+6a_{i,m}^2a_{j,m}^2-4a_{i,m}a_{j,m}^3+a_{j,m}^4)\\
&=\sum_{m=1}^n\sum_{i=1}^K\alpha_i a_{i,m}^4 - 4 \sum_{m=1}^n\biggl(\sum_{i=1}^K\alpha_i a_{i,m}^3\biggr)\biggl(\sum_{j=1}^K\alpha_j a_{j,m}\biggr)\\
& \phantom{{}={}} \quad + 6 \sum_{m=1}^n\biggl(\sum_{i=1}^K\alpha_i a_{i,m}^2\biggr)\biggl(\sum_{j=1}^K\alpha_j a_{j,m}^2\biggr) - 4 \sum_{m=1}^n\biggl(\sum_{i=1}^K\alpha_i a_{i,m}\biggr)\biggl(\sum_{j=1}^K\alpha_j a_{j,m}^3\biggr)\\
& \phantom{{}={}} \quad + \sum_{m=1}^n\sum_{j=1}^K\alpha_j a_{j,m}^4\text{,}
\end{align*}
which by \eqref{1} simplifies to
\begin{equation}\label{5}
\biggl(1-\sum_{i=1}^K\alpha_i^2\biggr)M^4 \geq 2\sum_{m=1}^n\sum_{i=1}^K\alpha_i a_{i,m}^4 + 6\sum_{m=1}^n\biggl(\sum_{i=1}^K\alpha_i a_{i,m}^2\biggr)^2\text{.}
\end{equation}
Similarly, if we apply $\displaystyle\sum_{j=1}^L\beta_j\sum_{\substack{i=1\\ i\neq j}}^L\beta_i$ to \eqref{3}, we obtain
\begin{equation}\label{6}
\biggl(1-\sum_{j=1}^L\beta_j^2\biggr)M^4 \geq 2\sum_{m=1}^n\sum_{j=1}^L\beta_j b_{j,m}^4 + 6\sum_{m=1}^n\biggl(\sum_{j=1}^L\beta_j b_{j,m}^2\biggr)^2\text{.}
\end{equation}
Next apply $\sum_{i=1}^K\alpha_i\sum_{j=1}^L\beta_j$ to \eqref{4}:
\begin{align*}
\mu^4 &=\sum_{i=1}^K\alpha_i\sum_{j=1}^L\beta_j \mu^4 \leq \sum_{i=1}^K\alpha_i\sum_{j=1}^L\beta_j\sum_{m=1}^n(a_{i,m}-b_{j,m})^4\\
&= \sum_{m=1}^n\sum_{i=1}^K\sum_{j=1}^L\alpha_i\beta_j(a_{i,m}^4-4a_{i,m}^3b_{j,m}+6a_{i,m}^2b_{j,m}^2-4a_{i,m}b_{j,m}^3+b_{j,m}^4)\\
&= \sum_{m=1}^n\biggl(\sum_{i=1}^K\alpha_i a_{i,m}^4\biggr)\biggl(\sum_{j=1}^L\beta_j\biggl) - 4\sum_{m=1}^n\biggl(\sum_{i=1}^K\alpha_i a_{i,m}^3\biggr)\biggl(\sum_{j=1}^L\beta_jb_{j,m}\biggr)\\
&\phantom{{}={}}\quad + 6\sum_{m=1}^n\biggl(\sum_{i=1}^K\alpha_i a_{i,m}^2\biggr)\biggl(\sum_{j=1}^L\beta_jb_{j,m}^2\biggr) - 4\sum_{m=1}^n \biggl(\sum_{i=1}^K\alpha_i a_{i,m}\biggr)\biggl(\sum_{j=1}^L\beta_jb_{j,m}^3\biggr)\\
&\phantom{{}={}}\quad + \sum_{m=1}^n\biggl(\sum_{i=1}^K\alpha_i\biggr)\biggl(\sum_{j=1}^L\beta_jb_{j,m}^4\biggr)\\
&\stackrel{\text{\eqref{1}}}{=} \sum_{m=1}^n\sum_{i=1}^K\alpha_i a_{i,m}^4 + 6\sum_{m=1}^n\biggl(\sum_{i=1}^K\alpha_i a_{i,m}^2\biggr)\biggl(\sum_{j=1}^L\beta_j b_{j,m}^2\biggr) + \sum_{m=1}^n\sum_{j=1}^L\beta_j b_{j,m}^4\text{,}
\end{align*}
that is,
\begin{equation}\label{7}
\sum_{m=1}^n\sum_{i=1}^K\alpha_i a_{i,m}^4 + \sum_{m=1}^n\sum_{j=1}^L\beta_j b_{j,m}^4 \geq \mu^4 - 6\sum_{m=1}^n\biggl(\sum_{i=1}^K\alpha_i a_{i,m}^2\biggr)\biggl(\sum_{j=1}^L\beta_j b_{j,m}^2\biggr)\text{.}
\end{equation}
Add \eqref{5} and \eqref{6} together:
\begin{align*}
&\phantom{{}\geq{}} \biggl(2-\sum_{i=1}^K\alpha_i^2-\sum_{j=1}^L\beta_j^2\biggr)M^4\\
&\geq 2\sum_{m=1}^n\sum_{i=1}^K\alpha_i a_{i,m}^4 + 2\sum_{m=1}^n\sum_{j=1}^L\beta_j b_{j,m}^4 +6\sum_{m=1}^n\biggl(\sum_{i=1}^K\alpha_i a_{i,m}^2\biggr)^2 + 6\sum_{m=1}^n\biggl(\sum_{j=1}^L\beta_j b_{j,m}^2\biggr)^2\\
&\stackrel{\text{\eqref{7}}}{\geq} 2\mu^4 - 12\sum_{m=1}^n\biggl(\sum_{i=1}^K\alpha_i a_{i,m}^2\biggr)\biggl(\sum_{j=1}^L\beta_j b_{j,m}^2\biggr)\\
&\quad + 6\sum_{m=1}^n\biggl(\sum_{i=1}^K\alpha_i a_{i,m}^2\biggr)^2 + 6\sum_{m=1}^n\biggl(\sum_{j=1}^L\beta_j b_{j,m}^2\biggr)^2\\
&= 2\mu^4 + 6\sum_{m=1}^n\Biggl(\biggl(\sum_{i=1}^K\alpha_i a_{i,m}^2\biggr)^2 - 2\biggl(\sum_{i=1}^K\alpha_i a_{i,m}^2\biggr)\biggl(\sum_{j=1}^L\beta_j b_{j,m}^2\biggr) + \biggl(\sum_{j=1}^L\beta_j b_{j,m}^2\biggr)^2\Biggr)\\
&= 2\mu^4 + 6\sum_{m=1}^n\biggl(\sum_{i=1}^K\alpha_i a_{i,m}^2-\sum_{j=1}^L\beta_j b_{j,m}^2\biggr)^2\\
&\geq 2\mu^4\text{.}
\end{align*}
Therefore,
\begin{equation}\label{8}
\frac{M^4}{\mu^4} \geq \frac{2}{2-\sum_{i=1}^K\alpha_i^2-\sum_{j=1}^L\beta_j^2}\text{.}
\end{equation}
By the Cauchy-Schwarz inequality and \eqref{1}, $\sum_{i=1}^K\alpha_i^2\geq 1/K$ and $\sum_{j=1}^L\beta_j^2\geq 1/L$.
Therefore,
\[ \sum_{i=1}^K\alpha_i^2 + \sum_{j=1}^L\beta_j^2 \geq \frac{1}{K} + \frac{1}{L} \geq
\begin{cases}
\frac{2}{n+2}+\frac{2}{n+2} & \text{if $n$ is even,}\\
\frac{2}{n+1}+\frac{2}{n+3} & \text{if $n$ is odd.}
\end{cases}
\]
Substitute this estimate into \eqref{8} to obtain
\[ \frac{M^4}{\mu^4} \geq
\begin{cases}
\displaystyle 1+\frac{2}{n} & \text{if $n$ is even,}\\
\displaystyle 1+\frac{2}{n-(n+2)^{-1}} & \text{if $n$ is odd,}
\end{cases}\]
which finishes the proof.
\end{proof}

\begin{proof}[Proof of Corollary~\ref{cor}]
It is well known and easy to see that for any $\vx\in\bR^n$, if $1\leq p\leq 4$ then $\norm{\vx}_4\leq\norm{\vx}_p\leq n^{1/p-1/4}\norm{\vx}_4$ and if $4\leq p < \infty$ then $\norm{\vx}_p\leq\norm{\vx}_4\leq n^{1/4-1/p}\norm{\vx}_p$.
Suppose that there exists an equilateral set $S$ of $n+2$ points in $\ell_p^n$.
Then 
\[ \frac{\displaystyle\max_{\vx,\vy\in S}\norm{\vx-\vy}_4}{\displaystyle\min_{\vx,\vy\in S, \vx\neq\vy}\norm{\vx-\vy}_4} \leq n^{\abs{1/4-1/p}}.\]
Combine this inequality with Theorem~\ref{maintheorem} to obtain
$1+\frac{2}{n} \leq n^{\abs{1-4/p}}$.
A calculation then shows that
\[\abs{p-4} \geq \frac{4\log(1+2/n)}{\log (n+2)} = \frac{8}{n\log n}\bigl(1+O(n^{-1})\bigr).\qedhere\] 
\end{proof}

\begin{proof}[Proof of Proposition~\ref{construction}]
Let $k\in\bN$, $x,y\in\bR$, and \[\va := (1+x,x,x,\dots,x)\in\ell_4^k\quad\text{and}\quad\vb:=(y,y,\dots,y)\in\ell_4^k.\]
We would like to choose $x$ and $y$ such that $\norm{\va}_4=\norm{\vb}_4$ and $\norm{\va-\vb}_4=2^{1/4}$.
This is equivalent to the following two simultaneous equations:
\begin{equation}\label{eq:sim}
\left.
\begin{aligned}
(1+x)^4+(k-1)x^4 &= k y^4\\
(1+x-y)^4+(k-1)(x-y)^4 &=2.
\end{aligned}
\right\}
\end{equation}
We postpone the proof of the following lemma.
\begin{lemma}\label{prop}
For each $k\in\bN$ the system \eqref{eq:sim} has a unique solution $(x_k,y_k)$ satisfying $y_k>0$.
Asymptotically as $k\to\infty$ we have
\[x_k=-k^{-1/2}+k^{-3/4}+O(k^{-1})\text{ and } y_k=k^{-1/4}-k^{-3/4}+O(k^{-1}).\]
\end{lemma}
Using the solution $(x,y)=(x_k,y_k)$ from the lemma, we obtain
\[\norm{\va}_4=\norm{\vb}_4=k^{1/4}y=1-k^{-1/2}+O(k^{-3/4}).\]
Write $\va_1,\dots,\va_k$ for the $k$ permutations of $\va$ and set $\va_{k+1}=\vb$.
Then \eqref{eq:sim} gives that $\set{\va_1,\va_2,\dots,\va_{k+1}}$ is equilateral in $\ell_4^k$.
Finally, let $n=2k$.
Then in the set \[S=\setbuilder{(\va_i,\vo)}{i=1,2,\dots,k+1}\cup\setbuilder{(\vo,\va_i)}{i=1,2,\dots,k+1}\] of $n+2$ points in $\ell_4^n$ the only non-zero distances are $2^{1/4}$ and $2^{1/4}\norm{\va}_4$.
Therefore, 
\[\frac{\displaystyle\max_{\vx,\vy\in S}\norm{\vx-\vy}_4}{\displaystyle\min_{\vx,\vy\in S, \vx\neq\vy}\norm{\vx-\vy}_4} = \frac{1}{\norm{\va}_4} = 1 + \sqrt{\frac{2}{n}}+O(n^{-3/4}).\]
The case where $n=2k+1$ is odd is handled by using the points $\va_1,\dots,\va_{k+1}\in\ell_4^k$ as constructed above and the analogous construction of $k+2$ points $\va'_1,\dots,\va'_{k+2}\in\ell_4^{k+1}$ satisfying $\norm{\va'_i-\va'_j}_4=2^{1/4}$ and $\norm{\va'_i}_4=1-(k+1)^{-1/2}+O(k^{-1})$.
Then the non-zero distances between points in \[S=\setbuilder{(\va_i,\vo)}{i=1,2,\dots,k+1}\cup\setbuilder{(\vo,\va'_i)}{i=1,2,\dots,k+2}\]
are $2^{1/4}$ and $\left(\norm{a_i}_4^4+\norm{a'_j}_4^4\right)^{1/4}$, giving the same asymptotics as before.
\end{proof}
\begin{proof}[Proof sketch of Lemma~\ref{prop}]
For $t\in\bR$ let
\[ f(t) = \left(\frac{(1+t)^4+(k-1)t^4}{k}\right)^{1/4} = k^{-1/4}\norm{(1,0\,\dots,0) + t(1,1,\dots,1)}_4.\]
Then \eqref{eq:sim} is equivalent to
$f(x)=\abs{y}$ and $f(x-y)=(2/k)^{1/4}$.
Since $\norm{\cdot}_4$ is a strictly convex norm, $f$ is strictly convex.
Since also $f(0)=k^{-1/4}$ and $\lim_{t\to\pm\infty}f(t)=\infty$, it follows that there is a unique $\alpha_k<0$ and a unique $\beta_k>0$ such that $f(\alpha_k)=f(\beta_k)=(2/k)^{1/4}$.
Thus, $x-y\in\set{\alpha_k,\beta_k}$.
It also follows that $f$ is strictly decreasing on $(-\infty,\alpha_k)$.
It is immediate from the definition that $f$ is strictly increasing on $(0,\infty)$. 
Since $f(-k^{-1/4})<(2/k)^{1/4}<f(k^{-1/4})$, it follows that $\alpha_k<-k^{-1/4}$ and $\beta_k<k^{-1/4}$.

By strict convexity of $\norm{\cdot}_4$, $f$ also satisfies the strict Lipschitz condition \[\abs{f(t+h)-f(t)}<h\quad\text{for all $t,h\in\bR$ with $h>0$.}\]
It follows that $t\mapsto f(t)-t$ is strictly decreasing and $t\mapsto f(t)+t$ is strictly increasing.
Since $\lim_{t\to\infty}(f(t)-t)=1/k$ and $\lim_{t\to-\infty}(f(t)+t)=-1/k$, it follows that $f(t)>t+1/k$ and for each $r>1/k$ there is a unique $t$ such that $f(t)-t=r$; also $f(t)>-t-1/k$ and for each $r>-1/k$ there is a unique $t$ such that $f(t)+t=r$.

We now consider the two cases $x-y=\alpha_k$ and $x-y=\beta_k$.

\smallskip
\noindent\textbf{Case I.}
If $x-y=\alpha_k$, then $f(x)=\abs{y}=\abs{x-\alpha_k}$.
Since $f(x)>-x-1/k\geq -x-k^{-1/4}>-x+\alpha_k$, necessarily $y=x-\alpha_k>0$ and $f(x)-x=-\alpha_k$.
Since $-\alpha_k>k^{-1/4}\geq 1/k$, there is a unique $x_k$ such that $f(x_k)-x_k=-\alpha_k$, and since $f(0)-0=k^{-1/4}<-\alpha_k$, it satisfies $x_k<0$.
Setting $y_k=x_k-\alpha_k$, we obtain that \eqref{eq:sim} has exactly one solution $(x_k,y_k)$ such that $x_k-y_k=\alpha_k$, and it satisfies $x_k<0<y_k$.

\smallskip
\noindent\textbf{Case II.}
If $x-y=\beta_k$, then we similarly obtain a unique solution $(x,y)$, this time satisfying $x<0$ and $y<0$.

\smallskip
Therefore, \eqref{eq:sim} has exactly two solutions, one with $y>0$ and one with $y<0$.
Next we approximate the solution $(x_k,y_k)$ of Case I.

From $f(\alpha_k)=(2/k)^{1/4}$ it follows that
\begin{equation}\label{alpha}
(1+\alpha_k)^4+(k-1)\alpha_k^4=2,
\end{equation}
which shows first that $\alpha_k=O(k^{-1/4})$ as $k\to\infty$, and then, since $\alpha_k<0$, that $\alpha_k=-k^{-1/4}+O(k^{-1/2})$.
We may rewrite \eqref{alpha} as
\begin{align}
\alpha_k &= -k^{-1/4}(1-4\alpha_k-6\alpha_k^2-4\alpha_k^3)^{1/4}\notag\\
&= -k^{-1/4}\left(1-\alpha_k-3\alpha_k^2-9\alpha_k^3+O(k^{-1})\right),\label{alpha2}
\end{align}
where we have used the Taylor expansion $(1+x)^{1/4}=1+\frac{1}{4}x-\frac{3}{32}x^2+\frac{7}{128}x^3+O(x^4)$.
Substitute the estimate $\alpha_k=-k^{-1/4}+O(k^{-1/2})$ into the right-hand side of \eqref{alpha2} to obtain the improved estimate $\alpha_k=-k^{-1/4}-k^{-1/2}+O(k^{-3/4})$, and again, to obtain
\begin{equation*}\label{alpha3}
\alpha_k=-k^{-1/4}-k^{-1/2}+2k^{-3/4}+O(k^{-1}).
\end{equation*}
Since
\[ f(-k^{-1/2})+k^{-1/2} = k^{-1/4}+k^{-1/2}-k^{-3/4}+O(k^{-1}) > -\alpha_k\]
for sufficiently large $k$, and $f(x_k)-x_k=-\alpha_k$, it follows that $x_k>-k^{-1/2}$ for large $k$, that is, $x_k=O(k^{-1/2})$.
It follows that \[f(x_k)-x_k = k^{-1/4}\left(1 + x_k +O(k^{-1})\right)-x_k.\]
Set this equal to $-\alpha_k$ and solve for $x_k$ to obtain
$x_k=-k^{-1/2}+k^{-3/4}+O(k^{-1})$
and $y_k=x_k-\alpha_k=k^{-1/4}-k^{-3/4}+O(k^{-1})$.
\end{proof}
\section*{Acknowledgement}
We thank the referee for helpful remarks that led to an improved paper.

\end{document}